\documentclass[a4paper,11pt]{amsart}

\usepackage[top=1.65in, bottom=1.65in, right=1.3in, left=1.3in]{geometry}

\usepackage{color}
\usepackage{xcolor}

\usepackage [utf8]{inputenc}
\usepackage[all]{xy}
\usepackage[english]{babel}
\usepackage{amssymb}
\usepackage[lite, initials]{amsrefs}

\newtheorem{teorema}{Theorem}[section]
\newtheorem*{theorem*}{Structure Theorem}
\newtheorem{lemma}[teorema]{Lemma}
\newtheorem{propos}[teorema]{Proposition}
\newtheorem{corol}[teorema]{Corollary}
\theoremstyle{definition}

\newtheorem{rem}[teorema]{Remark}
\newtheorem{defin}[teorema]{Definition}

\def\R{{\mathbb R}}

\def\C{{\mathbb C}}

\def\P{{\mathbb P}}

\newcommand{\Ci}{\mathcal{C}}

\def\Psh{{\rm Psh}}
\DeclareMathOperator{\spt}{spt}

\def\oli{\overline}				

\def\de{\partial}
\def\debar{\oli{\de}}

\title[On Minimal kernels and Levi currents]{On  Minimal kernels and Levi currents\\ on weakly
complete complex  manifolds}
\author[F.~Bianchi]{Fabrizio Bianchi\textsuperscript{1}}
\address{\textsuperscript{1}CNRS,  Univ.\ Lille, UMR 8524 - Laboratoire Paul Painlev\'e, F-59000 Lille, France}
\email{fabrizio.bianchi@univ-lille.fr}
\author[S.~Mongodi]{Samuele Mongodi\textsuperscript{2}}
\address{\textsuperscript{2}Politecnico di Milano, Dipartimento di Matematica, Via Bonardi, 9 -- I-20133 Milano, Italy}
\email{samuele.mongodi@polimi.it}

  \date{\today}

\begin{document}
\begin{abstract}
A  complex
manifold
$X$ is \emph{weakly complete}
 if it admits a continuous
plurisubharmonic
 exhaustion function $\phi$.
 The minimal kernels $\Sigma_X^k, k \in [0,\infty]$
 (the loci where are all $\Ci^k$  plurisubharmonic 
 exhaustion functions fail to be strictly plurisubharmonic),
introduced by Slodkowski-Tomassini,
and the Levi currents, introduced by Sibony, are both concepts
aimed at measuring
how far $X$ is from being Stein. 
We compare these
notions, prove that all Levi currents
are supported by all  the $\Sigma_X^k$'s,
 and give sufficient conditions for points in $\Sigma_X^k$ to be 
 in the support of some Levi current. 

 When $X$  is a surface and
   $\phi$ can be chosen analytic, 
building on previous work by the second author, Slodkowski, and Tomassini,
we prove the existence of a Levi current precisely
supported on $\Sigma_X^\infty$, and give a classification of
Levi currents on $X$. In particular,
unless $X$ is a modification of a Stein space, every point in
$X$ is in the support of some Levi current.
\end{abstract}

\maketitle

\section{Introduction}

Given an abstract (and possibly very complicated) manifold, a
natural question is whether it is possible to see 
it as a subset of a simpler space. In the real category, a fundamental
theorem by Nash states that this is always possible, and in a very strong
sense: every Riemannian manifold can be isometrically embedded in some $\mathbb{R}^N$.
When moving to the complex category, we can then ask
the following natural question: is it possible to embed any complex
manifold in some $\mathbb{C}^N$, by means of a holomorphic map?
We call \emph{Stein} a manifold for which the above holds true. This
time, the rigidity of holomorphic functions readily provides negative
examples: for instance, the maximum principle implies that any holomorphic
map on a compact complex manifold must be constant, and thus the manifold
cannot be Stein. A central question is then to understand
when
a given complex manifold is
Stein.
More specifically, given a dimension $n$, 
one would like to 
understand the \emph{obstructions} for an $n$-dimensional manifold to be Stein.

A major advance in this direction was provided
by
 Grauert \cite{Gra}: 
a complex manifold is Stein if and only if it admits a
$\Ci^2$
strictly plurisubharmonic
(psh for short) exhaustion function.
The $\Ci^2$ assumption was relaxed to $\Ci^0$ by
Narasimhan \cites{Na1,Na2}. 
In view of these results, 
it is natural to tackle the question
 by studying the positive cone 
 $\Psh_e^0(X)$
  of all continuous 
  psh exhaustion functions on
 $X$ (or more generally the cone
 $\Psh_e^k (X) := \Psh_e^0(X) \cap \mathcal{C}^k$ for some 
 $k \in [0,\infty]$, 
 and in particular to find obstructions for them to be strictly psh.
 As a rough idea, such obstructions must correspond to the presence of some sets in $X$ along
 which all continuous 
 psh functions must necessarily be pluriharmonic. As a prototypical example,
 the blow-up of a point and its corresponding exceptional divisor give precisely this kind of obstruction.

A precise study of this kind of phenomena was started
by Slodkowski
and Tomassini
in \cite{ST} in the setting of \emph{weakly complete complex manifolds}, i.e., manifolds admitting a 
continuous
psh exhaustion function.
 A crucial definition
 is the following: 
for $k\in [0,\infty]$
  the \emph{minimal kernel} of a manifold $X$ 
(with respect to $\Psh^k_e$) is 
 \begin{equation}\label{eq:sigmaX}
 \Sigma_X^k :=\{x\in X\ :\ i\de\debar u\textrm{ is degenerate at }x\, \forall u\in \Psh^k_e(X)\},
 \end{equation}
 i.e., the subset of $X$ where
 no  element of  $\Psh^k_e$ 
can be strictly psh. A key
 result  of \cite{ST} is that,
 whenever $\Psh_e^k (X)$ is not empty,
  there actually
  exists  a function $\phi_0\in \Psh^k_e (X)$
  (called \emph{minimal}) 
  which fails to be strictly psh \emph{precisely}
   on 
   the minimal kernel
$\Sigma_X^k$.
  Moreover, the minimal kernels are
   \emph{local maximum sets} (see Definition \ref{defi_localmax}). 
Some finer properties are also established (some requiring 
  at least the $\Ci^2$ regularity, see for instance \cite{ST}*{Theorem 3.9}). 
  Observe that the $\Sigma_X^k$'s are increasing in $k$, but it is not known whether
  equalities should occur 
  in general,
  see for instance
  \cite[Section 5.10]{Sl_pseudo}.

In \cites{mst,mst2}, the second author, Slodkowski, and Tomassini
showed that, if $X$ has complex dimension $2$ 
and $\Psh_e^\infty$
contains at least one real analytic function,
 the minimal kernel is either a union of
countably many compact (and negative) curves
or equal to the whole manifold, by
giving a full classification of the possible structures that such a manifold can present.
An important point here is that, although in general the
minimal kernel does not have a priori an analytic structure,
however its intersection with any level of a psh exhaustion function
 does (at least in dimension $2$).

In \cite{Sib}, Sibony introduced the notion of \emph{Levi current} (see Definition \ref{def_Levic}),
which is related to the (non-)existence of strictly psh functions on a complex manifold and thus to
the problem of determining whether a given
manifold is Stein, see also
\cite{OS, Sib_pf, Sib_pseudo}.
Extremal Levi-currents are supported
on sets where all 
continuous psh
 functions are constant.
In the case of infinitesimally homogeneous manifolds, a
foliation is constructed and linked to the obstructions to Steinness.

Our goal here is
to
compare these two approaches,
and in particular  to
use the notion of Levi current on $X$
 to study the analytic structure of
the minimal kernels $\Sigma_X^k$. 
 In order to do this, let
 us denote by $\Psh^k$, for $0\leq k\leq \infty$,
 the cone of $\Ci^k$ psh function on $X$ and
 define  the \emph{distribution}
$\mathcal E^k$ in $TX$ as
 \begin{equation}\label{eq:e} 
\mathcal{E}^{k}:=\{(x,v)\in TX\ :\ (d\varphi)_x(v)=0\ \forall\varphi \in \Psh^{k} (X) \}.
 \end{equation}
A \emph{distribution} is a subset of $TX$ whose intersection with 
$T_xX$ is a (real) vector subspace of the latter for every $x\in X$.
 In general, $\mathcal{E}^k$ will not be a subbundle of $TX$, as $\dim\mathcal{E}^k_x$ is not constant;
  however it is a closed subset of $TX$, hence the function $x\mapsto \dim \mathcal{E}^k_x$ is upper semicontinuous.

The following is our main result.

\begin{teorema}\label{t:main}
Let $X$ be a weakly complete complex
manifold
 and $T$ a Levi current on $X$. 
Denote by $\Sigma_X^k$ the minimal kernels of $X$ and by $F$ the union 
of the supports of all Levi currents on $X$.
Then $F \subseteq \Sigma_X^k$ for all $k\geq 0$
and
\begin{enumerate}
\item\label{item_t_locmax} if $T$ has compact support $K_T$, then $K_T$ is a local maximum set;
\item\label{item_t_KY}  if 
$K$
 is a local maximum set, then
 there exists a Levi current
supported on 
 $K$. In particular, $K \cap F \neq \emptyset$.
 \end{enumerate}
 Moreover, if $X$ is a surface,
$\phi \in \Ci^k$
a psh exhaustion function,  and
 $Y$ 
 a regular connected component of a level set $\{\phi=c\}$,
 \begin{enumerate}
 \setcounter{enumi}{2}
  \item\label{item_t_curves} if $4\leq k\leq \infty$ and  $U\subseteq \Sigma_X^k$
is an open set in $Y$   and there exists $x\in U$
 such that $\dim\mathcal{E}^k_x=2$,
  then $X$
is a union of compact complex curves. In particular, $F=\Sigma_X^k =X$;
\item\label{item_t_Y} if 
$2 \leq k \leq \infty$  and $Y \subseteq \Sigma_X^k$, then there exists $c'<c$
 such that 
 the connected component of $\phi^{-1} ([c',c])$ containing $Y$ is contained in $F$.
  In particular, $Y \subseteq F$.
\end{enumerate}
\end{teorema}

Moreover, given \emph{any}
psh function $u\in \Psh^0 (X)$,
any Levi current $T$ can be naturally disintegrated as $T= \int T_c \, d\mu$, where
$\mu$ is a positive measure on $\R$ and $T_c$ is a Levi current supported on $\{u=c\}$, see Corollary
\ref{cor_disint}.
This in particular gives examples of Levi currents for which Item \eqref{item_t_locmax} applies.
Notice that, whenever two level sets of $u,v\in \Psh^0(X)$
do not coincide, this allows to further refine the description of the extremal Levi current. This
motivates the definition \eqref{eq:e} of the
 distributions $\mathcal E^k$.

It follows from \cite{mst,ST} that, when $k\geq 2$,
for any level set $Y$ of an exhaustion function $\phi\in \Psh_e^0$,
 the set
$\Sigma_X^k\cap Y$ is a local maximum set (or empty), 
see Lemma \ref{l:intersection}. Hence Item \eqref{item_t_KY} applies for instance to such sets.
Finally, the manifold $X= \C \times \P^1$ provides an example where the Items \eqref{item_t_curves} and \eqref{item_t_Y} apply.

 \begin{rem}
 It would
be interesting to know if the equality holds in Item \eqref{item_t_KY} (and in particular for the
intersections between levels sets of a psh exhaustion function and the minimal kernel). Namely if, for \emph{any}
point in a local maximum set $K$, there exists a Levi current $T$ such that $x \in \spt T$.
\end{rem}

\medskip

The paper is organized as follows.
In Section \ref{s:levi} we recall the definition of Levi currents and 
the properties that we will need in the sequel. 
In Section \ref{section_locmax}
we prove
Items \eqref{item_t_locmax} 
 and \eqref{item_t_KY} 
of Theorem \ref{t:main}. The first item
is established for
$\Sigma_X^k \cap \{\phi_0=c_0\}$
(where $c_0$ is an attained value for a 
continuous
minimal function $\phi_0$ and $k\geq 2$) 
 in \cite{ST}*{Theorem 3.6},
and
 is actually a consequence of 
\cite{Sib_pf}*{Theorem 3.1},
 where it is proved through an integration by parts, see also
 \cite{Sib}*{Proposition 4.2}
 for an analogous statement 
  for $F$.
We give here a different proof by means of
a characterization of the
local maximum property
due to Slodkowski
\cite{Sl}
which allows to 
bypass the use of Brebermann functions and 
Jensen measures
as in \cite{ST}.
 In Section \ref{s:kernels} we study the relation between
the minimal kernels $\Sigma_X^k$ and
distributions in 
 the tangent bundle $TX$ given by 
directions satisfying some degeneracy 
condition. This leads to the proof
of Item \eqref{item_t_curves}.
The proof
of Theorem \ref{t:main} is completed in Section \ref{s:proof_main}, where
we establish Item \eqref{item_t_Y}. 
In Section \ref{s:analytic} we
consider
 the case where
 $X$ is a surface and the exhaustion
function in Theorem \ref{t:main} can be chosen analytic. By
exploting the main result in \cite{mst},
we deduce 
a classification of Levi currents in this case.

\subsection*{Acknowledgements}
The authors would like to thank Nessim Sibony
for the references
and for  very useful
comments that also helped to improve Item \eqref{item_t_KY} in the main theorem,
and Zbigniew Slodkowski
for
very
helpful remarks on a preliminary version of this paper.
They also would like to thank Vi\^{e}t-Anh Nguy\^{e}n and Giuseppe
Tomassini
for useful observations and discussions.

This work was supported by the 
Research in Pairs 2019 program of the CIRM (Centro Internazionale di Ricerca Matematica), Trento
and the FBK (Fondazione Bruno Kessler).
The authors would like to warmly thank CIRM-FBK for their support,
 hospitality, and
  excellent work conditions.
This project also 
received funding from the
I-SITE ULNE (ANR-16-IDEX-0004 ULNE),
the
LabEx CEMPI (ANR-11-LABX-0007-01)
 and from the CNRS 
 program PEPS JCJC 2019.

\section{Levi currents on complex manifolds}\label{s:levi}
In this section 
we recall the definition of Levi current and give the properties that we need in the sequel.
These results are essentially contained in \cite[Section 4]{Sib} and \cite[Section 3]{Sib_pseudo},
we sketch here the proofs for completeness.
We let
$X$ be any complex manifold and we denote by $\Psh^0(X)$ the space
of continuous plurisubharmonic functions on $X$.

\begin{defin}[Sibony \cite{Sib}]
 \label{def_Levic}A current $T$ on $X$ is called \emph{Levi current} if
\begin{enumerate}
\item $T$ is non zero;
\item $T$ is of bidimension $(1,1)$;
\item $T$ is positive;
\item $i\de\debar T=0$;
\item $T\wedge i\de\debar u=0$ for all $u\in \Psh^0(X)$.
\end{enumerate}
\end{defin}

A Levi current $T$ is \emph{extremal}
if $T=T_1 = T_2$ whenever $T=(T_1+T_2)/2$ for $T_1,T_2$ Levi currents.

\begin{lemma}\label{lemma_vanish} Take $u\in\Psh^0 (X)$ and
let
$T$
be
 a Levi current. The currents
$$T\wedge \de u\;,\quad \quad T\wedge\debar u\;,\quad  \mbox{and} \quad \quad T\wedge \de u\wedge \debar u$$
are all well defined and 
vanish identically on $X$.\end{lemma} 
\begin{proof}
The currents in the statement are well defined when $u$ is smooth, and the arguments from 
\cite{DS}*{Section 2} and
\cite{Sib}*{Section 4} prove the good definition for $u \in \Psh^0(X)$.

If $u\in \Psh^0(X)$, then also $\exp(u)\in\Psh^0(X)$. 
So, by Definition
\ref{def_Levic} 
of Levi current, 
$T\wedge i\de\debar \exp(u)$ vanishes identically. Hence, we have
$$0=\exp(u)T\wedge i\de\debar u
+i\exp(u)T\wedge\de u\wedge\debar u\;.$$
Given that $T\wedge i \de\debar u=0$, we conclude that 
$T\wedge \de u \wedge\debar u=0$.
This gives the last identity. 
We prove now the first one, the proof for the second one is similar.
Since $T$ is positive, 
for any $(0,1)$-form $\alpha$
by Cauchy-Schwarz's
inequality we get
\[
|\langle T, \partial u \wedge \alpha\rangle|^2 \leq 
c
\langle T, \partial u \wedge \bar \partial u\rangle
\cdot
\langle T, \alpha \wedge \bar \alpha\rangle 
\]
where $c$ is a constant independent of $T,u$, and $\alpha$. Since the 
first factor in the RHS is zero by the first part of the proof, the assertion follows.
\end{proof}

By a standard disintegration procedure, we obtain the following consequence.

\begin{corol}\label{cor_disint}Suppose $T$ is a Levi current and
 $u\in \Psh^0(X)$; then there exists a measure $\mu$ on $\R$ and a collection of currents $T_c$, $c\in\R$ such that
\begin{itemize}
\item $T_c$ is supported on $Y_c=\{x\in X\ :\ u(x)=c\}$ for all $c\in\R$;
\item $T_c$ is non zero for $\mu$-almost every $c\in\R$;
\item whenever $T_c\neq 0$, $T_c$ is a Levi current;
\item for every $2$-form $\alpha$ on $X$ we have
$$\langle T,\alpha\rangle=\int_{\R}\langle T_c,\alpha\rangle d\mu(c)\;.$$
\end{itemize}
Moreover, if $u\in\Psh^1(X)$ and $c$ is a regular value for $u$,
then $T_c=j_*S_c$, where
$j$ is the inclusion of $Y_c$ in $X$ and $S_c$ a current on the real manifold $Y_c$.
\end{corol}

Notice that $T_c$ needs not be extremal,
as it is easily seen considering $X=\C \times \P^1$.

\begin{rem}
Suppose now that $X$ is weakly complete  and let
$\phi$ be a psh exhaustion function.
By Corollary \ref{cor_disint}, every Levi current is
obtained by averaging Levi currents which are supported on the level sets
of $\phi$; as the latter is an exhaustion of $X$, its level sets are compact, so
every Levi current on $X$ is an integral average of compactly supported Levi
currents, i.e., positive currents of bidimension $(1,1)$ which
are $\de\debar$-closed and compactly supported.\end{rem}

\begin{corol}\label{corol_spsh}
If $T$ is a Levi current
and $u\in \Psh^1(X)$, then the vector field associated
to $T$ is tangent to the kernel of $\de u\wedge\debar u$,
whenever the latter is non-zero
(and the former is defined). 
Moreover, if there exists $v\in\Psh^0 (X)$
which is strictly plurisubharmonic at a point $x\in X$, then $x\not\in\spt T$
for any Levi current $T$.\end{corol}
\begin{proof}
The first statement is equivalent to $T\wedge\de u\wedge\debar u=0$, hence follows
from Lemma \ref{lemma_vanish}. 
Suppose now that
we have $v\in\Psh^0 (X)$ which is strictly psh at $x$
and a Levi current $T$.
First, by Richberg \cite[Satz 4.3]{Ri}
we can assume that $v$ is $\Ci^\infty$ and strictly psh near $x$.
Then, if $\rho\in\Ci^\infty_0(X)$ is supported in a neighbourhood
of $x$ where $v$ is strictly psh and $\|\rho\|_{\Ci^2}$ is small enough,
then also $v+\rho$ is
 psh.
 In particular, we can choose $\rho_i, 1\leq i \leq 2n$ such
that the
 $\ker(\de(v+\rho_i)\wedge\debar(v+\rho_i))_x$
 are independent (over $\mathbb R$).
This property holds true in a neighbourhood of $x$.
 Hence, as the vector field associated to $T$
(on a full measure subset of the support $T$ for the mass measure)
should belong to all
these subspaces, the only possibility
 is that $x\not\in\spt T$. This concludes the proof.\end{proof}

\begin{lemma}\label{lmm_spsh}
Let $T$ be a Levi current such that $K=\spt T$ is compact. If $u$
is defined and plurisubharmonic in an open neighbourhood  $V$ of $K$
 and strictly plurisubharmonic at  $x\in V$, 
 then $x\not\in\spt T$.\end{lemma}
\begin{proof}
Let
$V'\Subset V$ be an open neighbourhood of $K$ containing $x$.
Let  $\chi\in\Ci^\infty_0(V)$ be such that $\chi\vert_{V'}\equiv 1$, then $\chi u$ is
defined on $X$ and
psh on $V'$.
  As $\spt T\subseteq V'$,
also $\spt(T\wedge i\de\debar u)\subseteq V'$, so $T\wedge i\de\debar u$
 is positive; moreover, as $T$ is a Levi current,
we have $i \de\debar T=0$, hence
$$0=\langle i\de\debar T, u\rangle=\langle T, i\de\debar u\rangle=\langle T\wedge i \de\debar u, 1\rangle.$$
Therefore, $T\wedge i\de\debar u=0$ as a (positive) measure.

Since $(i\de\debar u)_x>0$,
this happens in a neighbourhood of $x$, so $T\wedge i\de\debar u$ is strictly
positive in a neighbourhood of $x$ unless $T$ is zero there. This gives
 $x\not\in\spt T$ and concludes the proof.
\end{proof}

\begin{lemma}\label{lmm_cptspp}
Suppose that a  current $T$ satisfies requests $1-4$ of Definition \ref{def_Levic}
and $T$ has compact support. 
Then $T$ is a Levi current.\end{lemma}
\begin{proof}
Given that $T$ is compactly supported, so are
$uT$, $T\wedge \de u$, $T\wedge \debar u$,
and $T\wedge i \de\debar u$
for all $u\in \Psh^0(X)$. Moreover, as $T$ is positive and $u$ is psh,
$T\wedge i \de\debar u$ is a positive measure on $X$; therefore, it is zero if and only if
$\langle T\wedge i \de\debar u, 1\rangle=0$.

Notice that, by Stokes' theorem, we have $\langle i \de\debar (uT), 1=0\rangle$, hence
$$0=-\langle i\de\debar u \wedge T, 1\rangle + \langle ui\de\debar T, 1\rangle +\langle i\de(\debar u\wedge  T), 1\rangle-\langle i\debar(\de u\wedge T), 1\rangle.$$
We have $i\de\debar T=0$ by hypothesis, while $\langle i\de(\debar u\wedge  T), 1\rangle$ and $\langle i\debar(\de u\wedge T), 1\rangle$ vanish by another application of Stokes' theorem. Therefore
$T\wedge i \de\debar u=0$,
that is, $T$ is a Levi current.
\end{proof}

\section{Local maximum sets}\label{section_locmax}

We establish here Items
\eqref{item_t_locmax}
 and \eqref{item_t_KY}
of Theorem \ref{t:main}.
We recall the following definition, see also \cite{ST}*{Section 2}
and \cite{R}.

\begin{defin}\label{defi_localmax}
Let $X$ be a complex manifold and
 $K \subset X$ be compact. We say that
\emph{$K$ is a local maximum set} if
 every $x\in K$ has a neighbourhood $U$
with the following property: for every compact set  $K' \subset U$ and every
function $\psi$ which is strictly psh in a neighbourhood of $K'$, we have
\[
\max_{K \cap K'}\psi = \max_{K \cap bK'} \psi.
\] 
\end{defin}

\begin{propos}\label{lemma:levi_local_max}
Suppose that $X$ is weakly complete. If
  $T$ is a Levi current with compact support, then $\spt T$ is a local maximum set.
 \end{propos}

\begin{proof}

Suppose that $K:=\spt T$ is not a local maximum set. 
By \cite{Sl}*{Proposition 2.3}
 there exist
$x\in K$,
a neighbourhood $B$ of $x$, with local
coordinates $z$ with origin in $x$,
 such that $B\equiv\{z\in\C^n\ :\ \|z\|<1\}$, and  
 $\psi\colon B \to \R$
  strictly psh
 with $\psi(x)=0$ and $\psi(y)\leq -\epsilon\|z(y)\|^2$ for all $y\in K\cap B$.
 Up to replacing $\psi$ by an element of a continuous approximating sequence, we can directly
 assume that $\psi$ is continuous.
By taking a possibly smaller ball a $x$, we can also assume that
$-\epsilon\|z(y)\|^2 -\epsilon/8 \leq \psi(y)$ on $K$.
 Set 
 \[A:= \{ y \in B \colon \|z(y)\|^2 > 3/4\}
\quad \mbox{ and } \quad
V=\{y\in B\ :\ |\psi(y)+\epsilon\|z(y)\|^2|<\epsilon/4\}.
\]
By the continuity of $\psi$ and the bounds above, 
$V$ is an open subset of $B$ containing $K$.
Consider $u\in\Psh^0(X)$ 
such that $u(x)=-\epsilon/4$ and $\sup_B|u|< \epsilon/2$
(this function exists because of the assumption on $X$).
Since $K\cap B\subset V$, there also exists
$\chi\in\Ci^\infty_0(X)$ be 
such that $\chi\vert_K\equiv 1$
and $\spt\chi\cap B\subset V$.
Define the function
$v\colon X \to \R$ as
\[
v=\begin{cases}
\chi \max\{u,\psi\} & \mbox{ on } B, \\
\chi u &  \mbox{ on } X\setminus B.
\end{cases}
\]
We claim that $v= \chi u$ on $A$.
Indeed,
for every $p\in \spt \chi \cap A$, we have 
that $p \in V$, and so
$\psi(p)<-3\epsilon/4+\epsilon/4=-\epsilon/2$. Hence,
$\psi(p)< u(p)$ and $v(p)= ( \chi u)(p)$.

By construction, $v$ coincides with
$\psi$ in a neighbourhood of $x$.
It follows that $v$
 is psh in a neighbourhood  of $K$ and
  strictly psh in $x$. Therefore, we have $x\not\in\spt T$
by Lemma \ref{lmm_spsh}. This gives a contradiction with the choice of $x\in K$ and
completes the proof.
\end{proof}

\begin{propos}\label{propos-locmax-supp-new}
Let $K\subset X$ be a local maximum set.
There exists a Levi current $T$ such that $\spt T \subseteq K$.
\end{propos}

\begin{proof}
By \cite[Theorem 3.1]{Sib_pseudo} (see also \cite[Section 4]{Sib})
and Lemma \ref{lmm_cptspp}, if there are no Levi currents supported on $K$, there exists a 
smooth strictly
psh function $u$
on some open neighbourhood $U$ of $K$. By slightly perturbing $u$, 
for every $x_0 \in K$ we can construct
$2n$ continuous strictly psh functions on some neighbourhood 
$K \subset U' \subseteq U$
such that $du_1, \dots, du_{2n}$ are linearly independent at $x_0$. This implies that,
in a neighbourhood  of $x_0$, we have
\[
\{u_1 = u_1 (x_0) \} \cap \dots \{u_{2n} = u_{2n}(x_0)\}= \{x_0\}.
\]

By \cite[Corollary 1.11]{Sl_pseudo} and \cite[Theorem 4.2]{Sl},
for every family 
 of continuous psh functions on $U'$
  there
exists a local maximum set $K'\subseteq K$ with the property that all functions of the family
are constant on $K'$. Choosing a point $x_0 \in K'$, the previous paragraph
 gives that $x_0$ is isolated in $K'$. 
 This is a contradiction, and the proof is complete.
\end{proof}

We conclude the section with the following result, that we will need to prove
Item
\eqref{item_t_curves} 
of Theorem \ref{t:main}.

\begin{lemma}\label{l:intersection}
Let $X$ be a weakly complete complex surface and $Y$ a regular level for a 
$\Ci^0$
exhaustion function $\phi$. Then, 
for all $k\geq 2$,
$\Sigma_X^k \cap Y$ is a local maximum set and, for all 
local maximum sets $K\subseteq \Sigma_X^k$,
$K$
is
 foliated by  holomorphic discs, i.e., it is
 locally a union of disjoint holomorphic discs.
\end{lemma}

\begin{proof}
The intersection $\Sigma_X^k \cap Y$ is a local maximum set
by \cite{mst}*{Theorem 3.2}. As observed in
 \cite{mst}*{Proposition 3.5}, the proof of the
 lemma is then essentially 
  given in
\cite{ST}*{Lemma 4.1}, see
in particular
the Assertion in the proof of that lemma.
\end{proof}

We point out that \cite[Lemma 4.1]{ST} 
relies on a result by Shcherbina \cite{Sh},
 which holds true only in dimension 2; this is the reason for
 restricting ourselves to the case of surfaces. it would be interesting
 to prove (or disprove) a similar statement in higher dimension.

\section{Kernels and tangent directions}
\label{s:kernels}

In this section we let $X$ be a
weakly complete complex manifold of
dimension $n$ and
assume that
 $\Psh^k(X)$ contains at least one exhaustion function $\phi:X\to\R$ for some $2\leq k\leq \infty$.
Recall that 
 the minimal kernel $\Sigma_X^k$
  of $X$ 
 is
 defined as in \eqref{eq:sigmaX}
 and the distribution $\mathcal E^k$ of $TX$
 as in \eqref{eq:e}.
We consider further the distribution
 $\mathcal{S}^k$ of $TX$ given by
\begin{equation}	\label{eq_s}
\mathcal{S}^k:=\{(x,\xi)\in TX\ :\ \xi\in T_xX,\ (i\de\debar u)_x(\xi,\xi)=0\, \forall u\in \Psh^k (X)\}\; .
\end{equation}
Similar objects have already appeared in relation to the study  of the Levi problem, see for instance
\cite{Hir} in the case of homogeneous manifolds and  \cite{ST,HST}.
We also 
set
\[
E^k_{\ell} := \{x\in X \colon\dim \mathcal{E}^k_x \geq \ell \}
\quad  \mbox{ and } \quad
S^k_1 := \{x\in X \colon\dim \mathcal{S}^k_x \geq 1\}.
\]
By definition,
$E^k_{\ell} \subseteq E^k_{\ell-1}$ and $E^k_{\ell}$ is closed in $E^k_{\ell-1}$ for all $\ell \geq 1$.
Observe moreover that $\mathcal S^k$  is a \emph{complex} distribution.

\begin{rem}
Let $T$ be a Levi current. Then, for almost every point of the support
of $T$ (with respect to the mass measure), the 
vector field associated to $T$ at $x$ belongs to the fibre
 $\mathcal S^k_x$ of $\mathcal S^k$ at $x$.
\end{rem}

\begin{propos}\label{prop_es}
We have $\mathcal{S}^k \subseteq \mathcal {E}^k$, and 
$S^k_1 = E^k_2 = E^k_1 = \Sigma_X^k$.
\end{propos}

\begin{proof}
It follows
from the definition
 of $\mathcal E^k$ that $E^k_1 = E^k_2$. Moreover,
 if $(x,v)\not\in\mathcal{E}^k$,
 there exists $\psi\in\Psh^{k}(X)$ 
 such that $(d\psi)_x(v)\neq 0$; then
$$i\de\debar \exp(\psi)=\exp(\psi)i\de\debar\psi+i \exp(\psi)\de\psi\wedge\debar\psi\geq 0,$$
which implies that
 $(i\de\debar\exp(\psi))_x(v,v)\geq \exp(\psi(x))|\de\psi_x(v)|^2>0$ and so
 $(x,v)\not\in \mathcal{S}^k$. It follows that $\mathcal S^k \subseteq \mathcal E^k$.

 We now prove that $E^k_1 = \Sigma_X^k$.
If $x\not\in\Sigma_X^k$, then there exists
$\psi \in \Psh^k_e (X)$
which is strictly psh around $x$; therefore, given any
$\rho:X\to\R$ smooth with compact support near $x$, there exists $\epsilon>0$
such
that $\psi+\epsilon\rho$ is still psh.
 So, we can construct
psh functions  of class $\Ci^k$ 
whose
 differentials span the tangent space at $x$, which implies that these differentials
do not
have any nontrivial common kernel in $T_xX$. So $\mathcal{E}_x^k=\{0\}$, hence
$E_1^k \subseteq \Sigma_X^k$.

On the other hand, if $\mathcal{E}^k_x=\{0\}$, given $v_1,\ldots,v_{2n}\in T_xX$
linearly independent, we can choose
 psh functions
 $\psi_{ij}$, $i,j=1,\ldots, 2n$
 of class $\Ci^k$ and
 such that
$(d\psi_{ij})_x(v_i+v_j)\neq 0$. Therefore, the function
$\psi:=\sum_{i,j=1}^{2n} \psi_{ij}^2$
has positive defined Levi form at $x$. 
Adding to the exhaustion
 function $\phi$ suitable multiples of $\psi$, we see that 
 $x\not\in\Sigma_X^k$.
 This
gives $E_1^k \supseteq \Sigma_X^k$, hence $E_1^k=\Sigma_X^k$.

 In order to
conclude, we need to prove that
$S_1^k\supseteq \Sigma_X^k$.
Take $x \in \Sigma_X^k$  and
suppose  by contradiction 
that, for every $v\in T_xX$ there is $\varphi_v:X\to\R$
which is
 $\Ci^k$, psh, and such that $(dd^c\varphi_v)_x(v,v)>0$.
Then, as above, 
we
 can construct a  $\Ci^k$ function $\psi$
which is strictly psh
at $x$. This gives the desired contradiction
and completes the proof.
\end{proof}

The following result
gives Item \eqref{item_t_curves} of our main Theorem
\ref{t:main}.

\begin{propos}\label{t:e2}
Let $X$ be a weakly complete complex surface, $\phi$ a 
$\Ci^k$, $4\leq k\leq \infty$,
exhaustion psh function and 
$Y$ a regular connected component of a level set
$\{\phi=c\}$ 
of $\phi$.
Suppose that $U\subseteq Y$ 
is an open set in $Y$ and $U\subseteq\Sigma_X^k$. If there exists $x\in U$
 such that $\dim\mathcal{E}_x^k=2$,
  then $X$
is a union of compact complex curves.
In particular, $\Sigma_X^k=F=X=E_2^k$,
and 
$E_3^k$
 is contained in a (possibly empty) 
analytic subset of  the
 singular levels for $\phi$.
\end{propos}

We will need the
 following theorem by Nishino, see \cite{Ni}*{Proposition 9 and Théorème II}
and \cite{MT}*{Section 2.2.1}.

\begin{teorema}[Nishino]
\label{t_nishino}
Let $X$ be a weakly complete or compact surface 
that contains an uncountable family $\mathcal F$ 
of disjoint connected compact complex curves.
Then there exist a Riemann surface $R$ and a meromorphic map $h \colon X\to R$ with compact fibers.
\end{teorema}

\begin{proof}[Proof of Proposition \ref{t:e2}]
By Theorem \ref{t_nishino}, to prove the first assertion it is enough to show
that $X$ contains uncountably many disjoint compact complex curves.

Since $x\in U$ is such that $\dim\mathcal{E}_x^k=2$, 
 there exists $\psi \in \Psh_e^k (X)$
 such that $(d\phi)_x$ and $(d\psi)_x$ 
 are linearly independent;
 hence,
the map $\psi\vert_Y:Y\to\R$
 is not constant.
Since $k\geq 4$, by Sard's theorem
 we can find regular values $b$ for $\psi$ arbitrarily close
 to $b_0:=\psi(x)$, therefore the sets $C_b=\{y\in Y\ :\ \psi(y)=b\}$
 intersect the open set $U\subseteq \Sigma_X^k$.

For any $y\in C_b\cap \Sigma_X^k$, by Proposition \ref{prop_es}
we have $T_yC_b=\mathcal{E}^k_y =\mathcal S^k_y$.
 Therefore, $C_b\cap\Sigma_X^k$ is a complex curve, being a real, smooth 
 $2$-dimensional manifold with complex tangent space.
On the other hand, the set  $C_b\setminus\Sigma_X^k$ is open in $C_b$. Let $z\in C_b$
be a boundary point (with respect to $C_b$); 
 as $z\in\Sigma_X^k$, by Lemma \ref{l:intersection}
 there is a holomorphic disc $f:\mathbb{D}\to X$
such that $f(\mathbb{D})\subset Y$ and $f(0)=z$.
 If $\zeta\in\mathbb{D}$ is close enough to $0$,
then, setting $w=f(\zeta)$, we have $w\in\Sigma_X^k$, and
$(d\phi)_w$ and $(d\psi)_w$ independent. This gives 
$w\in E_2^k\setminus E_3^k$, which in turn implies that 
$\mathcal{E}_w^k =T_wC_{\psi(w)}$. Therefore $f(\mathbb{D})$ coincides locally with a leaf $C_{b'}$. Hence $C_b$
 is contained in $\Sigma_X^k$, so it is a compact complex curve.

As $b$ was taken arbitrarily  among the regular values close enough to $b_0$, we find uncountably many 
disjoint (since they correspond to distinct values)
compact complex curves in $X$, 
as desired. 

In order to conclude, we need to prove the final assertion on $E_3^k$.
We proved above
 that there exists a meromorphic map $h\colon X\to R$ with compact fibres, where
 $R$ is Riemann surface. It is enough to prove that
 $E_3^k \subseteq \{h'=0\}$.
 
 Let $x\in X$ be such that $h'(x)\neq 0$.
Consider a strictly psh 
exhaustion function $\psi$ for $R$
 (which we can assume to be $\Ci^\infty$ near $x$ by \cite{Ri})
 and the family of functions
$\mathcal F :=\{ \psi + \epsilon \rho\}$, where 
$\rho$ is a smooth function compactly supported near $h(x)$. For every such $\rho$,
$\psi + \epsilon \rho$ is still strictly psh for $\epsilon$ sufficiently small. Thus, 
we can obtain
a set of generators 
for the tangent space given by
differentials at $h(x)$ of psh functions in $\mathcal F$.
Pre-composing the corresponding functions with $h$, we obtain that the space
of differentials at $x$
 of psh  
  functions on $X$ has dimension at least $2$. Hence, $x \notin E_3^k$, and the
 proof is complete.
\end{proof}

\begin{rem}
 Suppose that $X$ is a surface
 and $Y$ a regular level for an
exhaustion function $\phi\in \Psh^0_e(X)$.
Let 
$K\subseteq Y\cap\Sigma_X$ be a local maximum set. 
By Lemma \ref{l:intersection}, 
$Y$ and $K$ are foliated by holomorphic discs.
For every such disk,  its tangent bundle is exactly the restriction of
$\mathcal S$. By \cite{BrSib}*{Theorem 1.4}, there exists
 a $\de\debar$-closed
positive current of bidimension $(1,1)$, directed by $\mathcal S$, supported in $K$.
By Lemma \ref{lmm_cptspp}, such current is a Levi current. This gives a different proof
of Item \eqref{item_t_KY} when $\dim X=2$.
\end{rem}

\section{End of the proof of Theorem \ref{t:main}}\label{s:proof_main}

It follows
from Corollary \ref{corol_spsh} (or Lemma \ref{lmm_spsh})
 that
 $\spt T\subseteq \Sigma_X^k$ for every Levi current $T$
 and all $k\geq 0$.
Thus, we have $F \subseteq \Sigma_X^k$ for all $k\geq 0$. 
 Moreover,
Items \eqref{item_t_locmax}, \eqref{item_t_KY}, and
  \eqref{item_t_curves}
 follow from Propositions
  \ref{lemma:levi_local_max}, \ref{propos-locmax-supp-new}, and
  \ref{t:e2}, respectively.

Let now $Y$ be a regular 
connected component of a  level set for an
exhaustion psh function $\phi\in \Psh^k_e(X)$ for some $k\geq  2$.
The remaining item
follows
from the next proposition.

\begin{propos}\label{prop-levels-K}
 If $k\geq 2$  and 
 $Y\subseteq \Sigma_X^k$, 
 there exists $c'<c$ such that
 the connected component of $\phi^{-1}([c',c])$
 containing $Y$ is contained in $F$.
 \end{propos}
\begin{proof}
We 
 assume for simplicity that the level $\{\phi=c\}$ is regular and connected, the argument is similar
otherwise.
Since $k\geq 2$, by \cite{ST}*{Theorem 3.9} there is $c'<c$ such that, setting
$$K=\{x\in X\ :\ c'\leq \phi(x)\leq c\},$$
the
form
$(dd^c\phi)^2$ 
vanishes
 on the interior of
 $K$, hence on $K$.
So, we have $K \subseteq \Sigma_X^k$.

 Consider the current $T$ given by
 \[
 T:= i\de\phi\wedge\debar\phi.
\] 
It is clear that $T$ is
a current of bidimension $(1,1)$,
positive and
 directed by the complex subspace of the tangent
of the levels of $\phi$. Moreover, $i\de\debar T$ is induced by 
the form
\[i\de\debar
 (i\de\phi\wedge\debar\phi)
=
-(\de\debar\phi)^2.
\]
So, $i\de\debar T$
 vanishes
 where $\phi$ is not strictly psh, hence on $\Sigma_X^k$. 
 Let $B$ be
the interior of $K$, then the restriction of $T$ to $B$ is a current of bidimension $(1,1)$, positive,
$\de\debar$-closed (in $B$); moreover, given $u\in \Psh^0(X)$, we have that
$T\wedge\de\debar u=0$ on $\Sigma_X^k$, so $T$ is a Levi current.

By construction and Lemma \ref{lemma_vanish}
we have $T\wedge \de\phi=0$, so we
can disintegrate $T$ along the levels of $\phi$, see Corollary \ref{cor_disint}:
there exist 
currents $T_s$ with $s\in (c',c)$, such that, for $\alpha$ a $2$-form with $\spt \alpha\subset B$, 
$$\langle T,\alpha\rangle=\int_{c'}^c\langle T_s,\alpha\rangle\, d\mu(s)
\quad
\mbox{ for some measure } \mu \mbox{ on } (c',c).$$ 
Since 
$\phi\in \Ci^2$,
the measure $\mu$ is absolutely
continuous with respect to the Lebesgue measure on $(c',c)$.

As $T$ is $\de\debar$-closed in $B$, so is $\mu$-almost every $T_s$ in $B$;
therefore, for a dense open set of $s\in (c',c)$, $T_s$ is
a positive, $\de\debar$-closed current of bidimension $(1,1)$ and
$$\spt T_s=\{x\in X\ :\ \phi(x)=s\}.$$
The set in the RHS is compact since $\phi$ is an exhaustion function. 
By Lemma \ref{lmm_cptspp}, $T_s$
 is a Levi current.

In conclusion, the level set $\{x\in X\ :\ \phi(x)=s\}$ 
is contained in $F$ for almost all $s\in (c',c)$, so $\phi^{-1}([c',c])\subseteq F$, 
as $F$ is closed. In particular, $Y\subseteq F$.
\end{proof}

The proof of Theorem \ref{t:main} is complete.

\section{Real analytic exhaustion function}\label{s:analytic}

A classification of 
those
  weakly complete complex surfaces
 $X$
  admitting
   an analytic exhaustion function is given in \cite{mst}. As a direct consequence, we can get an analogous
   complete classification of the possible Levi currents in this setting.
   
   First notice that each exceptional divisor $V$ in $X$
    corresponds to an extremal Levi current given by the current of integration $[V]$. 
    Without loss of generality, to simplify our next statement, we
    can thus assume that $X$ has no such divisors on the regular levels of $\alpha$.
     The statement for a general  $X$ is then a direct consequence.

\begin{teorema}
Let $X$ be a weakly complete complex surface admitting an 
analytic exhaustion function $\alpha$. Assume that $X$ has no exceptional divisors on the regular levels of $\alpha$.
Then one of the following possibilities hold:
\begin{enumerate}
\item $X$ is Stein (and so, admits no Levi currents);
\item $F=\Sigma_X^{\infty} = X= \cup V_i$, where all the $V_i$ are (disjoint) connected
compact curves, and all extremal Levi currents are of the form $\lambda [V_i']$ for some positive $\lambda$, with $V_i'$ an irreducible component of some $V_i$;
\item $F=\Sigma_X^{\infty} =X$, every regular level $Y_c$ of $\alpha$ is foliated by curves $U_i$, and 
the support of any extremal Levi currents on $Y_c$ is equal to 
(a connected component of) $Y_c$.
\end{enumerate}
\end{teorema}

Observe also that, although a priori we would only have  $\Sigma_X^k \subseteq \Sigma_X^\infty$ for all $k\geq 0$,
 the above geometric description implies that
 $\Sigma_X^k = \Sigma_X^\infty$
 for all $k\geq 0$.

\begin{proof}
It follows from \cite{mst}*{Theorem 1.1} that
one of the following possibilities holds:
\begin{enumerate}
\item $X$ is a Stein space;
\item $X$ is proper over a (possibly singular) complex curve;
\item the connected components of the regular levels of $\alpha$
are foliated with dense complex curves.
\end{enumerate}
In the first and second cases, the assertion
 follows from the characterization of Levi currents given in Section \ref{s:levi}.
In the third case, a Levi current can be constructed, for instance, by means of \cite{BrSib}*{Theorem 1.4}. 
By proposition \ref{lemma:levi_local_max}, the support of any Levi current
is a local maximum set.
By \cite{MZ}*{Lemma 3.3}, a local maximum set contained in a Levi-flat hypersurface must be a union of leaves of the Levi foliation.

 Hence, in the third case, 
any Levi current on a regular level set of the exhaustion function
 is supported on the whole level set, as all the leaves of the Levi foliation are dense.
This in particular applies to extremal Levi currents. The proof is complete.
\end{proof}


\begin{bibdiv}
\begin{biblist}

\bib{BrSib}{article}{
   author={Berndtsson, Bo},
   author={Sibony, Nessim},
   title={The $\overline\partial$-equation on a positive current},
   journal={Inventiones Mathematicae},
   volume={147},
   date={2002},
   number={2},
   pages={371--428},
   issn={0020-9910},
   doi={10.1007/s002220100178},
}


\bib{DS}{article}{
  title={Pull-back of currents by holomorphic maps},
  author={Dinh, Tien-Cuong},
  author={Sibony, Nessim},
  journal={Manuscripta Mathematica},
  volume={123},
  number={3},
  pages={357--371},
  year={2007},
  publisher={Springer},
  doi={10.1007/s00229-007-0103-5},
}


\bib{Gra}{article}{
  author={Grauert, Hans},
  title={On Levi's problem and the imbedding of real-analytic manifolds},
  journal={Annals of Mathematics},
  volume={68},
  number={2},
  pages={460--472},
  year={1958},
  doi={10.2307/1970257},
}


\bib{HST}{article}{
  title={On defining functions for unbounded pseudoconvex domains},
  author={Harz, Tobias}
  author={Shcherbina, Nikolay},
  author={Tomassini, Giuseppe},
  journal={Mathematische Zeitschrift},
  volume={286}, 
  pages={987--1002},
  year={2017},
  doi={10.1007/s00209-016-1792-9},
}


\bib{Hir}{article}{
  title={Le probleme de L{\'e}vi pour les espaces homogenes},
  author={Hirschowitz, Andr{\'e}},
  journal={Bulletin de la Soci{\'e}t{\'e} Math{\'e}matique de France},
  volume={103},
  pages={191--201},
  year={1975},
  doi={10.24033/bsmf.1801},
}



\bib{M}{article}{
	author={Mongodi, S.},
	title={Weakly complete domains in Grauert type surfaces},
	journal={Annali di Matematica Pura ed Applicata (1923 -)},
	volume={198},
  number={4},
  pages={1185--1189},
  year={2019},
  doi={10.1007/s10231-018-0814-0},
	}
	
	
\bib{MZ}{article}{
author={Mongodi, S.}, 
author={Slodkowski, Z.},
title={Domains with a continuous exhaustion in weakly complete surfaces},
 journal={Mathematische Zeitschrift},
 volume={296}, 
 pages={1011–1019},
 year={2020}, 
 doi={10.1007/s00209-020-02466-z},
 }
 
 
\bib{crass}{article}{
author={Mongodi, Samuele},
   author={Slodkowski, Zbigniew},
   author={Tomassini, Giuseppe},
title = {On weakly complete surfaces},
journal = {Comptes Rendus Mathematique},
volume = {353},
number = {11},
pages = {969 -- 972},
year = {2015},
doi = {10.1016/j.crma.2015.08.009},
}


  \bib{mst}{article}{
     author = {Mongodi, Samuele},
     author = {Slodkowski, Zbigniew},
		author = {Tomassini, Giuseppe},
      title = {Weakly complete complex surfaces},
    journal = {Indiana University Mathematics Journal},
     volume = {67},
       year = {2018},
     number = {2},
      pages = {899 -- 935},
      doi={10.1512/iumj.2018.67.6306},
}
	
	
	\bib{mst2}{article}{
author = {Mongodi, Samuele},
     author = {Slodkowski, Zbigniew},
		author = {Tomassini, Giuseppe},
title = {Some properties of Grauert type surfaces},
journal = {International Journal of Mathematics},
volume = {28},
number = {8},
pages = {1750063 (16 pages)},
year = {2017},
doi = {10.1142/S0129167X1750063X},
}



\bib{MT}{incollection}{
  title={Minimal kernels and compact analytic objects in complex surfaces},
  author={Mongodi, Samuele},
  author={Tomassini, Giuseppe},
  booktitle={Advancements in Complex Analysis},
  pages={329--362},
  year={2020},
 doi={10-1007/978-3-030-40120-7{\_}9},
}


\bib{Na1}{article}{
  title={The Levi problem for complex spaces},
  author={Narasimhan, Raghavan},
  journal={Mathematische Annalen},
  volume={142},
  number={4},
  pages={355--365},
  year={1961},
  doi={10.1007/BF01451029},
}


\bib{Na2}{article}{
  title={The Levi problem for complex spaces II},
  author={Narasimhan, Raghavan},
  journal={Mathematische Annalen},
  volume={146},
  number={3},
  pages={195--216},
  year={1962},
  doi={10.1007/BF01470950},
}


\bib{Ni}{article}{
   author={Nishino, Toshio},
   title={L'existence d'une fonction analytique sur une vari\'{e}t\'{e} analytique
   complexe \`a deux dimensions},
   language={French},
 journal={Publications of the Research Institute for Mathematical Sciences},
   volume={18},
   date={1982},
   number={1},
   pages={387--419},
   issn={0034-5318},
   doi={10.2977/prims/1195184029},
}

\bib{OS}{article}{
  title={Bounded psh functions and pseudoconvexity in K{\"a}hler manifold},
  author={Ohsawa, Takeo},
  author={Sibony, Nessim},
  journal={Nagoya Mathematical Journal},
  volume={149},
  pages={1--8},
  year={1998},
  doi={10.1017/S0027763000006516},
}


\bib{Ri}{article}{
  title={Stetige streng pseudokonvexe Funktionen},
  author={Richberg, Rolf},
  journal={Mathematische Annalen},
  volume={175},
  number={4},
  pages={257--286},
  year={1967},
  doi={10.1007/BF02063212},
}



\bib{R}{article}{
  title={The local maximum modulus principle},
  author={Rossi, Hugo},
  journal={Annals of Mathematics},
  volume={72},
  number={1},
  pages={1--11},
  year={1960},
  doi={10.2307/1970145},
}

\bib{Sh}{article}{
  title={On the polynomial hull of a graph},
  author={Shcherbina, Nikolay},
  journal={Indiana University Mathematics Journal},
  volume={42},
  number={2},
  pages={477--503},
  year={1993},
  doi={10.1512/iumj.1993.42.42022},
}




\bib{Sib_pf}{article}{
  title={Pfaff systems, currents and hulls},
  author={Sibony, Nessim},
  journal={Mathematische Zeitschrift},
  volume={285},
  number={3-4},
  pages={1107--1123},
  year={2017},
  publisher={Springer},
  doi={10.1007/s00209-016-1740-8},
}


\bib{Sib}{article}{
	title={Levi problem in complex manifolds},
  author={Sibony, Nessim},
  journal={Mathematische Annalen},
  date={2018},
  volume={371},
  pages={1047--1067},
	doi={10.1007/s00208-017-1539-x},
}

\bib{Sib_pseudo}{article}{
  title={Pseudoconvex domains with smooth boundary in projective spaces},
  author={Sibony, Nessim},
  journal={Mathematische Zeitschrift},
  year={2020},
  publisher={Springer},
  doi={10.1007/s00209-020-02613-6},
}



\bib{Sl}{article}{
  title={Local maximum property and q-plurisubharmonic functions in uniform algebras},
  author={Slodkowski, Zbigniew},
  journal={Journal of mathematical analysis and applications},
  volume={115},
  number={1},
  pages={105--130},
  year={1986},
  doi={10.1016/0022-247X(86)90027-2},
}


\bib{Sl_pseudo}{incollection}{
  title={Pseudoconcave decompositions in complex manifolds},
  author={Slodkowski, Zbigniew},
  booktitle={Contemporary Mathematics, Advances in Complex geometry}
  volume={31},
  pages={239--259},
  year={2019},
  doi={10.1090/conm/735/14829},
}

\bib{ST}{article}{
    author={Slodkowski, Zibgniew},
   author={Tomassini, Giuseppe},
   title={Minimal kernels of weakly complete spaces},
   journal={Journal of Functional Analysis},
   volume={210},
   date={2004},
   number={1},
   pages={125--147},
   doi={10.1016/S0022-1236(03)00182-4},
}
  \end{biblist}
\end{bibdiv}
\end{document}